\documentclass[12pt]{amsart}
\usepackage{amsthm}
\usepackage{amssymb}
\usepackage{geometry}
\usepackage{enumerate}
\usepackage{listings}
\usepackage{setspace}
\usepackage{pdflscape}
\usepackage{tikz}
\usepackage{tikz-cd}
\usepackage[unicode=true]
 {hyperref}
\hypersetup{
colorlinks=true,
urlcolor=black,
citecolor=blue,
linkcolor=blue,
}
\makeatletter
\newcommand{\addresseshere}{%
  \enddoc@text\let\enddoc@text\relax
}
\makeatother

\allowdisplaybreaks

\newtheorem{thm}{Theorem}[section]
\newtheorem{prop}[thm]{Proposition}
\newtheorem{lem}[thm]{Lemma}

\theoremstyle{definition}

\theoremstyle{remark}
\newtheorem{remark}[thm]{Remark}

\numberwithin{equation}{section}

\newcommand{\ff}{\mathfrak{f}}
\newcommand{\fg}{\mathfrak{g}}
\newcommand{\Perm}{\mathrm{Perm}}
\newcommand{\Hom}{\mathrm{Hom}}
\newcommand{\Aut}{\mathrm{Aut}}
\newcommand{\Inn}{\mathrm{Inn}}
\newcommand{\Out}{\mathrm{Out}}
\newcommand{\Hol}{\mathrm{Hol}}
\newcommand{\res}{\mathrm{res}}
\newcommand{\ep}{\epsilon}
\newcommand{\E}{\mathcal{E}}

\newcommand{\Cent}{\mathrm{Cent}}

\begin{document}

\large 

\title{Hopf-Galois structures on a Galois $S_n$-extension}

\author{Cindy (Sin Yi) Tsang}
\address{School of Mathematics, Sun Yat-Sen University, Zhuhai}
\email{zengshy26@mail.sysu.edu.cn}\urladdr{http://sites.google.com/site/cindysinyitsang/} 

\date{\today}

\maketitle

\begin{abstract}In this paper, we shall determine the exact number of Hopf-Galois structures on a Galois $S_n$-extension, where $S_n$ denotes the symmetric group on $n$ letters.
\end{abstract}

\tableofcontents

\vspace{-5mm}

\section{Introduction}

Let $L/K$ be a finite Galois extension with Galois group $G$. Write $\Perm(G)$ for the symmetric group of $G$. Recall that a subgroup $\mathcal{N}$ of $\Perm(G)$ is said to be \emph{regular} if the map
\[ \xi_\mathcal{N}: \mathcal{N}\longrightarrow G;\hspace{1em}\xi_\mathcal{N}(\eta) = \eta(1)\]
is bijective, or equivalently, if the $\mathcal{N}$-action on $G$ is both transitive and free. For example, the images of the left and right regular representations
\[\begin{cases}
\lambda: G\longrightarrow \Perm(G);\hspace{1em}\lambda(\sigma) = (x\mapsto \sigma x),\\
\rho:G\longrightarrow\Perm(G);\hspace{1em}\rho(\sigma) = (x\mapsto x\sigma^{-1}),
\end{cases}\]
respectively, are both regular subgroups of $\Perm(G)$. By work of C. Greither and B. Pareigis in \cite{GP}, there is an explicit one-to-one correspondence between Hopf-Galois structures on $L/K$ and elements in 
\[ \E(G) = \{\mbox{regular subgroups of $\Perm(G)$ normalized by $\lambda(G)$}\}. \]
In turn, this set may be written as a disjoint union of the subsets
\[ \E(G,N) = \left\{ \begin{array}{c}\mbox{regular subgroups of $\Perm(G)$ which are}
\\ \mbox{isomorphic to $N$ and normalized by $\lambda(G)$}\end{array}\right\},\]
where $N$ ranges over all groups of order $|G|$, up to isomorphism. Hence, it is of interest to enumerate $\E(G)$ and $\E(G,N)$. A useful tool is the formula
\begin{equation}\label{B formula}\#\E(G,N) = \frac{|\Aut(G)|}{|\Aut(N)|}\cdot \#\left\{\begin{array}{c}\mbox{regular subgroups in $\Hol(N)$}\\\mbox{which are isomorphic to $G$}\end{array}\right\},\end{equation}
where $\Hol(N)$ denotes the \emph{holomorph} of $N$ and is given by
\begin{equation}\label{Hol(N)}\Hol(N)=\rho(N)\rtimes \Aut(N).\end{equation}
This was shown by N. P. Byott in \cite{By96}; or see \cite[Section 7]{Childs book}. We shall refer the reader to \cite[Chapter 2]{Childs book} for more background on Hopf-Galois structures.

\vspace{1.5mm}

For each $n\in\mathbb{N}$, we shall use the standard notation:
\begin{align*}
S_n & = \mbox{the symmetric group on $n$ letters},\\
A_n & = \mbox{the alternating group on $n$ letters},\\
C_n & = \mbox{the cyclic group of order $n$}.
\end{align*}
The purpose of this paper is to determine the exact size of $\E(S_n,N)$ for each group $N$ of order $n!$ and hence the size of $\E(S_n)$. For $n=1,2$, trivially
\begin{align*} \#\E(S_1) &= \#\E(S_1,S_1) = 1, \\\#\E(S_2) &= \#\E(S_2,S_2) = 1.\end{align*}
For $n=3$, by \cite[Corollary 6.5]{Byott pq}, we already know that
\[ \#\E(S_3) = \#\E(S_3,S_3) + \#\E(S_3,C_6) = 2 + 3 = 5.\]
For $n=4$, using the formula (\ref{B formula}) and the \textsc{magma} code $1$ in the appendix, we may compute that:

\begin{thm}\label{thm4}
Let $N$ be a group of order $4!$. Then, we have
\[ \#\E(S_4,N) = \begin{cases}
8 & \mbox{for $N\simeq S_4$},\\ 36 &\mbox{for $N\simeq A_4\times C_2$}, \\ 24 & \mbox{for $N\simeq S_3\times C_2\times C_2$}, \\ 48 & \mbox{for $N\simeq C_6\times C_2\times C_2$},\\ 0 & \mbox{otherwise}.
\end{cases}\] 
In particular, we have $\#\E(S_4) = 116$.
\end{thm}

\begin{remark}The first and last cases in Theorem~\ref{thm4} were known previously, by \cite[Theorem 7]{Childs simple} and \cite[Proposition 4]{AnSn}, respectively.
\end{remark}

For $n\geq 5$, it was computed in \cite[Corollaries 6 and 10]{Childs simple}, respectively, that
\begin{align}\label{n count}
\#\E(S_n,S_n) &= 2\cdot\sum_{\substack{0\leq k\leq n/2\\k\small \mbox{ is even}}}\frac{n!}{(n-2k)!\cdot 2^k\cdot k!}, \\\notag
\#\E(S_n,A_n\times C_2) &=2\cdot\sum_{\substack{0\leq k\leq n/2\\k\small \mbox{ is odd}}}\frac{n!}{(n-2k)!\cdot 2^k\cdot k!}.\end{align}
Also, the case when $n=6$ is slightly different because $S_6$ has an exceptional outer automorphism. In \cite[p. 91]{Childs simple}, it was shown that 
\begin{equation}\label{6 count}\#\E(S_6,M_{10}) = 72 \mbox{ and }\#\E(S_6,\mbox{PGL}(2,9)) =0,\end{equation}
where $M_{10}$ denotes the Mathieu group of degree $10$. Our main result is that $\#\E(S_n,N) = 0$ for all other choices of $N$. More precisely, we shall prove:

\begin{thm}\label{thm56}For $n\geq 5$, let $N$ be a group of order $n!$. Then, we have
\[ \#\E(S_n,N) = 0 \mbox{ for }N\not\simeq \begin{cases}S_n, A_n\times C_2 &\mbox{if $n\neq 6$}, \\ S_6, A_6\times C_2, M_{10},\mathrm{PGL}(2,9)&\mbox{if $n=6$}.\end{cases}\]
In particular, we have
\[\#\E(S_n) = \begin{cases}2\cdot\sum\limits_{0\leq k\leq n/2}\dfrac{n!}{(n-2k)!\cdot 2^k\cdot k!}&\mbox{if $n\neq6$},\\ 224 &\mbox{if $n=6$}.\end{cases}\]
\end{thm}


\section{Hopf-Galois structures on a Galois $S_n$-extension for $n\geq 5$}

In this section, assume that $n\geq 5$, in which case $A_n$ is non-abelian simple, and let $N$ be a group of order $n!$. In  \cite[Theorem 1.7]{Tsang HG}, the author gave some necessary conditions on $N$ in order for $\#\E(S_n,N)$ to be non-zero. Below, we shall refine the arguments there to prove Theorem~\ref{thm56}.

\vspace{1.5mm}

As already noted in \cite[Proposition 2.1]{Tsang HG}, which is a consequence of (\ref{Hol(N)}), a subgroup of $\Hol(N)$ isomorphic to $S_n$ is of the shape 
\begin{equation}\label{subgroup fg}\{\rho(\fg(\sigma))\cdot\ff(\sigma):\sigma\in S_n\},\mbox{ where }\begin{cases}\ff\in\Hom(S_n,\Aut(N))\\\fg\in\mbox{Map}(S_n,N)\end{cases}\end{equation}
satisfy the relation
\begin{equation}\label{fg relations} \fg(\sigma\tau) = \fg(\sigma)\cdot\ff(\sigma)(\fg(\tau))\mbox{ for all }\sigma,\tau\in S_n. \end{equation}
Moreover, the subgroup (\ref{subgroup fg}) is regular precisely when $\fg$ is bijective. Hence, by (\ref{B formula}), if $\#\E(S_n,N)$ is non-zero, then we may choose $\fg$ to be bijective. In view of (\ref{n count}) and (\ref{6 count}), to prove Theorem~\ref{thm56}, it then suffices to show that
\begin{equation}\label{thm'}
\mbox{$\fg$ is bijective}\implies N\simeq \begin{cases}S_n \mbox{ or }A_n\times C_2 &\mbox{if $n\neq 6$}, \\ S_n \mbox{ or } A_n\times C_2 \mbox{ or }M_{10}\mbox{ or }\mathrm{PGL}(2,9)&\mbox{if $n=6$}.\end{cases}
\end{equation}
To that end, given any group $\Gamma$, we shall use the following notation:
\begin{align*}
\Inn(\Gamma) & = \mbox{the inner automorphism group of $\Gamma$},\\
\Out(\Gamma) & = \mbox{the outer automorphism group of $\Gamma$},\\
Z(\Gamma) & = \mbox{the center of $\Gamma$},\\
[\Gamma,\Gamma] & = \mbox{the commutator subgroup of $\Gamma$}.
\end{align*}
Recall that $\Gamma$ is said to be \emph{perfect} if $\Gamma = [\Gamma,\Gamma]$. The implication (\ref{thm'}), and in particular Theorem~\ref{thm56}, is a direct consequence of the propositions below.

\begin{prop}\label{prop1}Suppose that $\fg$ is bijective. Then, we have:
\begin{enumerate}[(a)]
\item The group $N$ cannot be perfect.
\item The group $N$ contains a copy of $A_n$.
\end{enumerate}
\end{prop}

\begin{prop}\label{prop2}Suppose that $N$ contains a copy of $A_n$. Then, we have:
\[ N\simeq \begin{cases}
S_n\mbox{ or }A_n\times C_2 &\mbox{if $n\neq6$},\\
S_6\mbox{ or }A_6\times C_2\mbox{ or }M_{10}\mbox{ or }\mathrm{PGL}(2,9)&\mbox{if $n=6$}.
\end{cases}\]
\end{prop}

Hence, it remains to prove Propositions~\ref{prop1} and~\ref{prop2}, which we shall do in the subsequent subsections.

\subsection{Ruling out the groups $N$ which are perfect} In what follows, suppose that $\fg$ is bijective, and we shall prove Proposition~\ref{prop1} (a). Let us remark that this is the hardest step to the proof of Theorem~\ref{thm56}.

\vspace{1.5mm}

Suppose for contradiction that $N$ is perfect. Note that then $N$ cannot have a subgroup of index two, because such a subgroup is normal, and its quotient is abelian. Also, it shall be helpful to recall that $A_n$ is the unique non-trivial proper normal subgroup of $S_n$, since $n\geq 5$.

\begin{lem}\label{lem1}The homomorphism $\ff$ is injective.
\end{lem}
\begin{proof}Notice that $\fg$ restricts to a homomorphism $\ker(\ff)\longrightarrow N$ by (\ref{fg relations}). This implies that $\ker(\ff)$ cannot contain $A_n$, since $N$ has no subgroup of index two, and hence $\ker(\ff)$ must be trivial.
\end{proof}

Recall that the double cover of $A_n$ is the unique group $2A_n$, up to isomorphism, fitting into a short exact sequence
\[\begin{tikzcd}[column sep = 1cm] 1\arrow{r}& C_2 \arrow{r}& 2A_n \arrow{r} & A_n \arrow{r}& 1\end{tikzcd}\]
such that the image of $C_2$ in $2A_n$ lies in $Z(2A_n)$ and $[2A_n,2A_n]$. It is known that $2A_n$ is a perfect group whose center has order two.

\begin{lem}\label{lem2}We have $N\simeq 2A_n$ and $\ff(A_n) \subset\Inn(N)$. 
\end{lem}
\begin{proof}Let $M$ be any proper and maximal characteristic subgroup of $N$. The quotient group $N/M$ is then characteristically simple, and so we have
\[ N/M \simeq T^m,\mbox{ where $T$ is a simple group and $m\in\mathbb{N}$.}\]
Since $N$ is perfect, necessarily $T$ is non-abelian. Also, it is known that
\[ \Aut(T^m) \simeq \Aut(T)^m\rtimes S_m,\]
by \cite[Lemma 3.2]{Byott simple}, for example. Note that there is a natural homomorphism
\begin{equation}\label{Aut mod M}\Aut(N)\longrightarrow\Aut(N/M);\hspace{1em}\varphi\mapsto(\eta M\mapsto\varphi(\eta)M)\end{equation}
because $M$ is characteristic. We then have a homomorphism
\[\begin{tikzcd}[column sep = 1cm]\overline{\ff}: S_n \arrow{r} & \Aut(N) \arrow{r} &\Aut(N/M) \arrow{r}{\simeq} & \Aut(T)^m\rtimes S_m\end{tikzcd}\]
induced by $\ff$. Replacing $2A_n$ by $A_n$ in the proof of \cite[Lemma 4.8]{Tsang HG},  the exact same argument shows that the homomorphism
\[\begin{tikzcd}[column sep = 1.5cm]
A_n \arrow{r}{\overline{\ff}}& \Aut(T)^m\rtimes S_m \arrow{r}{\mbox{\tiny projection}} & S_m
\end{tikzcd}\]
must be trivial, and in turn the homomorphism
\[\begin{tikzcd}[column sep = 1.5cm]
A_n \arrow{r}{\overline{\ff}}& \Aut(T)^m \arrow{r}{\mbox{\tiny quotient}} &\Out(T)^m
\end{tikzcd}\]
must be trivial as well. Hence, we conclude that $\overline{\ff}(A_n)$ lies in $\Inn(T)^m$. Let us remark that the argument in \cite[Lemma 4.8]{Tsang HG} uses \cite[Proposition 3.4]{Tsang HG}, which is a consequence of the classification of finite simple groups.

\vspace{1.5mm}

Next, consider the surjective map
\[\begin{tikzcd}[column sep = 1cm]\overline{\fg} : S_n\arrow{r}& N \arrow{r}& N/M \arrow{r}{\simeq}& T^m\end{tikzcd}\]
induced by $\fg$. As shown in \cite[Lemma 4.1]{Tsang HG}, the relation (\ref{fg relations}) implies that:
\begin{itemize}
\item $\overline{\fg}$ restricts to a homomorphism $\ker(\overline{\ff})\longrightarrow T^m$,
\item $\fg^{-1}(M)$ is a subgroup of $S_n$ of index $[N:M]$.
\end{itemize}
The latter in turn implies that $A_n\not\subset\fg^{-1}(M)$, for otherwise
\[ 2 = [S_n :A_n] \geq [S_n:\fg^{-1}(M)] = [N:M] = |T|^m,\]
which is impossible because $T$ is non-abelian. We consider two possibilities:
\begin{enumerate}[(1)]
\item If $\ker(\overline{\ff})\supset A_n$, then $\overline{\fg}$ restricts to a homomorphism $A_n\longrightarrow T^m$, which cannot be trivial because $A_n\not\subset \fg^{-1}(M)$, and so must be an embedding.
\item If $\ker(\overline{\ff})=1$, the $\overline{\ff}$ restricts to an embedding $A_n\longrightarrow \Inn(T)^m\simeq T^m$.
\end{enumerate}
In both cases, we have $ |A_n| \leq |T|^m$ and hence $|M|\leq2$. Necessarily $|M|=2$, for otherwise $N\simeq T^m$, and $A_n$ would embed into $N$ as a subgroup of index two. It follows that $|A_n| = |T|^m$, whence
\[ A_n\simeq T^m,\mbox{ so in fact $m=1$, and we have }N/M\simeq T\simeq A_n.\]
Since any normal subgroup of order two lies in the center, by the uniqueness of the double cover of $A_n$, we see that $N\simeq 2A_n$ and $M = Z(N)$.

\vspace{1.5mm}

Since $M = Z(N)$, the map (\ref{Aut mod M}) is injective, by \cite[Proposition 3.5 (c)]{Tsang HG}, for example. It follows that $\Out(N)$ embeds into $\Out(N/M)\simeq \Out(A_n)$ and so is abelian. Thus, the homomorphism
\[\begin{tikzcd}[column sep = 1.5cm]
A_n \arrow{r}{\ff}& \Aut(N) \arrow{r}{\mbox{\tiny quotient}} & \Out(N)
\end{tikzcd}\]
must be trivial, whence $\ff(A_n)$ lies in $\Inn(N)$.
\end{proof}

From Lemmas~\ref{lem1} and~\ref{lem2}, we deduce that $\ff$ induces an isomorphism
\[f\in\Hom(A_n,N/Z(N));\hspace{1em}f(\sigma) = \widetilde{f}(\sigma)Z(N),\]
where $\widetilde{f}:A_n\longrightarrow N$ is any injective map satisfying
\[\ff(\sigma)(\eta) = \widetilde{f}(\sigma)\cdot\eta\cdot\widetilde{f}(\sigma)^{-1}\mbox{ for all }\sigma\in A_n\mbox{ and }\eta\in N.\]
Also, let $\zeta\in S_n$ be the element such that $\fg(\zeta)$ is the generator of $Z(N)$.

\begin{lem}\label{lem3}We have $\zeta \in A_n$.
\end{lem}
\begin{proof}Using (\ref{fg relations}), it is easy to check that the map
\[ g : A_n \longrightarrow N/Z(N);\hspace{1em} g(\sigma) = \fg(\sigma)\widetilde{f}(\sigma) Z(N)\]
is a homomorphism. It cannot be trivial, for otherwise by picking a different $\widetilde{f}$ if necessary, we may assume that $\fg(\sigma) = \widetilde{f}(\sigma)^{-1}$ for all $\sigma\in A_n$. But then again by (\ref{fg relations}), for all $\sigma,\tau\in A_n$, we have
\[ \widetilde{f}(\sigma\tau)^{-1} = \fg(\sigma\tau) = \fg(\sigma)\cdot \ff(\sigma)(\fg(\tau)) = \widetilde{f}(\tau)^{-1}\widetilde{f}(\sigma)^{-1},\]
which means that $\widetilde{f}$ is in fact a homomorphism. This is impossible since $\widetilde{f}$ is injective and $N$ cannot have any subgroup of index two. We have thus shown that $g$ must be an isomorphism. Put $\varphi = f^{-1}\circ g$, which is an automorphism on $A_n$. If $\zeta\notin A_n$, then for any $\sigma\in A_n$, we would have
\[ \varphi(\sigma) = \sigma \implies f(\sigma) = g(\sigma) \implies \fg(\sigma) \in Z(N) \implies \sigma = 1,\]
namely $\varphi$ is fixed point free. But $A_n$ has no such automorphism since $n\geq5$. Hence, we must have $\zeta\in A_n$, as claimed.
\end{proof}

\begin{lem}\label{lem4}For any $\sigma \in S_n$, we have 
\[\sigma\zeta = \zeta\sigma\mbox{ if and only if }\widetilde{f}(\zeta)\fg(\sigma) = \fg(\sigma)\widetilde{f}(\zeta).\]
Moreover, the element $\zeta$ has order two.
\end{lem}
\begin{proof}By (\ref{fg relations}), we have
\[\fg(\sigma\zeta) = \fg(\sigma)\cdot \ff(\sigma)(\fg(\zeta))\mbox{ and }
\fg(\zeta\sigma)  = \fg(\zeta)\cdot \ff(\zeta)(\fg(\sigma)).\]
Since $Z(N)$ has order two, its generator $\fg(\zeta)$ has to be fixed by all automorphisms on $N$. We then see that $\fg(\zeta^2) = \fg(\zeta)^2 = 1$ as well as
\[ \fg(\sigma\zeta) = \fg(\zeta\sigma) \mbox{ if and only if } \ff(\zeta)(\fg(\sigma)) = \fg(\sigma).\]
The claim then follows since $\zeta\in A_n$ by Lemma~\ref{lem3} and $\fg$ is bijective.
\end{proof}

By Lemmas~\ref{lem3} and~\ref{lem4}, we know that $\zeta$ has order two and lies in $A_n$. We may then assume, without loss of generality, that
\[ \zeta = (1\hspace{2mm}2)(3\hspace{2mm}4)\cdots (4r-1\hspace{2mm}4r)\]
Also, recall from Lemma~\ref{lem2} that $N\simeq 2A_n$, and write
\[ \widetilde{\zeta} = [1\hspace{2mm}2][3\hspace{2mm}4]\cdots [4r-1\hspace{2mm}4r] \] 
for some lift of $\zeta$ in $2A_n$, where the notation is as in \cite[Chapter 2.7.2]{Wilson}. Given any element $\gamma$ in a group $\Gamma$, let $\Cent_\Gamma(\gamma)$ denote the centralizer of $\gamma$ in $\Gamma$.

\begin{lem}\label{lem5}We have
\[\#\Cent_{S_n}(\zeta) = 2\cdot\#\Cent_{A_n}(\zeta) \mbox{ and } \#\Cent_{2A_n}(\widetilde{\zeta}) = \#\Cent_{A_n}(\zeta).\] 
\end{lem}
\begin{proof}The first equality is obvious because $(1\hspace{2mm}2)\notin A_n$ commutes with $\zeta$. As for the second equality, note that the quotient map induces a homomorphism
\[ \Cent_{2A_n}(\widetilde{\zeta}) \longrightarrow \Cent_{A_n}(\zeta)\]
which is $2$-to-$1$ and whose image has index at most two. But
\[ \zeta(1\hspace{2mm}3)(2\hspace{2mm}4)\zeta^{-1} = (1\hspace{2mm}3)(2\hspace{2mm}4) \mbox{ while }\widetilde{\zeta}[1\hspace{2mm}3][2\hspace{2mm}4]\widetilde{\zeta}^{-1} = z[1\hspace{2mm}3][2\hspace{2mm}4]\]
in the notation of \cite[Chapter 2.7.2]{Wilson}, where $z\in Z(2A_n)$ is non-trivial. This means that the index is in fact equal to two, from which the claim follows.
\end{proof}

We are now ready to prove Proposition~\ref{prop1} (a). From Lemma~\ref{lem4} and the fact that $f$ is an isomorphism, we deduce that
\[ \#\Cent_{S_n}(\zeta) = \#\Cent_{N}(\widetilde{f}(\zeta)) = \#\Cent_{2A_n}(\widetilde{\zeta}).\]
But this contradicts Lemma~\ref{lem5}, whence $N$ cannot be perfect, as desired.

\subsection{Reducing to the groups $N$ which contain a copy of $A_n$} In what follows, suppose that $\fg$ is bijective, and we shall prove Proposition~\ref{prop1} (b)

\vspace{1.5mm}

By Proposition~\ref{prop1} (a), the group $N$ is not perfect, so it has a proper and maximal characteristic subgroup $M$ containing $[N,N]$. The quotient $N/M$ is then abelian, and by the proof of the second statement of \cite[Theorem 1.7]{Tsang HG}, we know that $A_n = \fg^{-1}(M) $. This means that $\fg$ restricts to a bijective map
\[ \res(\fg): A_n \longrightarrow M;\hspace{1em}\res(\fg)(\sigma) = \fg(\sigma).\]
Also, since $M$ is characteristic, the map $\ff$ induces a homomorphism
\[ \res(\ff) : A_n \longrightarrow \Aut(M);\hspace{1em}\res(\ff)(\sigma) = \ff(\sigma)|_M.\]
Clearly, it follows directly from (\ref{fg relations}) that
\[ \res(\fg)(\sigma\tau) = \res(\fg)(\sigma)\cdot (\res(\ff)(\sigma))(\res(\fg)(\tau))\mbox{ for all }\sigma,\tau\in A_n.\]
Analogous to (\ref{subgroup fg}), this implies that $\Hol(M)$ has a regular subgroup isomorphic to $A_n$. Since $A_n$ is non-abelian simple, we deduce from \cite{Byott simple} that $M\simeq A_n$. This proves that $N$ contains a copy of $A_n$, as desired.

\subsection{Classifying group extensions of $C_2$ by $A_n$} In what follows, suppose that $N$ contains a copy of $A_n$, and we shall prove Proposition~\ref{prop2}. Since any subgroup of index two is necessarily normal, the hypothesis implies that $N$ fits in a short exact sequence
\begin{equation}\label{SES}\begin{tikzcd}[column sep = 1cm] 1\arrow{r}& A_n \arrow{r}& N \arrow{r}& C_2 \arrow{r}& 1\end{tikzcd}\end{equation}
We shall write $A$ for the image of $A_n$ in $N$ and $\ep$ for the non-trivial element in $C_2$. Let us make two observations.


\begin{lem}\label{lem split} The short exact sequence (\ref{SES}) splits if $n\neq6$.
\end{lem}
\begin{proof}For each $\eta\in N$, note that we have an automorphism
\[ \varphi_{\eta}\in\Aut(A); \hspace{1em} \varphi_{\eta}(x) = \eta\cdot x\cdot \eta^{-1}.\]
Fix an element $\eta_0\in N\setminus A$. Suppose now that $n\neq6$, so then $\Aut(A)\simeq S_n$, with $\Inn(A)$ corresponding to $A_n$. 
\begin{enumerate}[(1)]
\item If $\varphi_{\eta_0}\notin\Inn(A)$, then there exists $a\in A$ such that $\varphi_{\eta_0}\varphi_{a}$ has order two, which means that $(\eta_0a)^2$ centralizes $A$. 
\item  If $\varphi_{\eta_0}\in\Inn(A)$, then there exists $a\in A$ such that $\varphi_{\eta_0}\varphi_a$ is the identity, which means that $\eta_0a$ centralizes $A$.
\end{enumerate}
In both cases, since $(\eta_0a)^2 \in A$ and $Z(A)=1$, we deduce that $\eta_0a$ has order two. Thus, by sending $\ep\mapsto \eta_0a$, we obtain a homomorphism $C_2\longrightarrow N$ which splits the short exact sequence (\ref{SES}), as desired.
\end{proof}

\begin{lem}\label{lem semi}Let  $\psi:C_2\longrightarrow \Aut(A_n)$ be a homomorphism such that 
\[ \psi(\ep) : A_n\longrightarrow A_n;\hspace{1em}\psi(\ep)(x) = \sigma_0\cdot x\cdot\sigma_0^{-1}\]
for some $\sigma_0\in S_n$. Then, we have
\[ A_n\rtimes _\psi C_2 \simeq \begin{cases}
S_n & \mbox{if }\sigma_0\notin A_n,\\
A_n\times C_2 &\mbox{if }\sigma_0\in A_n.
\end{cases} \]
\end{lem}
\begin{proof}Note that $\sigma_0$ must have order dividing two. Also, we have
\[  (\tau_1,\ep^{i_1})(\tau_2,\ep^{i_2}) = (\tau_1\cdot \sigma_0^{i_1}\tau_2\sigma_0^{-i_1},\ep^{i_1+i_2})\mbox{ in } A_n\rtimes_\psi C_2\]
for all $\tau_1,\tau_2\in A_n$ and $i_1,i_2\in\mathbb{Z}$. It is then straightforward to verify that
\[ A_n\rtimes_\psi C_2 \longrightarrow S_n;\hspace{1em}(\tau,\ep^i)\mapsto \tau\sigma_0^i\]
defines an isomorphism if $\sigma_0\notin A_n$, and 
\[ A_n\rtimes_\psi C_2 \longrightarrow A_n\times C_2;\hspace{1em}(\tau,\ep^i)\mapsto (\tau\sigma_0^i,\ep^i)\]
defines an isomorphism if $\sigma_0\in A_n$. This proves the claim.
\end{proof}

We are now ready to prove Proposition~\ref{prop2}. For $n=6$, the claim may be verified using the \textsc{magma} code $2$ in the appendix. For $n\neq6$, the claim is a direct consequence of Lemmas~\ref{lem split} and~\ref{lem semi} because then $\Aut(A_n)\simeq S_n$.

\section{Acknowledgments}

The author would like to thank the referee for some helpful comments. She would also like to recognize the software \textsc{magma} \cite{magma} and \textsc{gap} \cite{GAP} which were used in the computations to deal with the cases $n=4,6$.

\addresseshere

\begin{landscape}
\section*{Appendix: \textsc{magma} codes}

\textsc{magma} code $1$:
\begin{lstlisting}[
  mathescape,
  columns=fullflexible,
  basicstyle=\ttfamily,
]
G:=SymmetricGroup(4); 
AutG:=AutomorphismGroup(G);
NN:=SmallGroups(24);
for i in [1..#NN] do
N:=SmallGroup(24,i);
AutN:=AutomorphismGroup(N);
Hol:=Holomorph(N);
RegSub:=RegularSubgroups(Hol);
L:=[0]; //  sizes of the conjugacy classes of regular subgroups in Hol(N) isomorphic to G
  for R in RegSub do
    if IsIsomorphic(R`subgroup,G) then
    Append(~L,R`length);
    end if;
  end for;
  E:=(#AutG/#AutN)*&+L; // formula (1.1)
  if E ne 0 then
    print <i,E>;
  end if;
end for;
\end{lstlisting}

\textsc{magma} code $1$ output:
\begin{lstlisting}[
  mathescape,
  columns=fullflexible,
  basicstyle=\ttfamily,
]
<12,8> // SmallGroup(24,12) = S4
<13,36> // SmallGroup(24,13) = A4 x C2
<14,24> // SmallGroup(24,14) = S3 x C2 x C2
<15,48> // SmallGroup(24,15) = C6 x C2 x C2
// The identifications of SmallGroup(24,i) here may be checked using
the StructureDescription() or IdGroup() command in $\textsc{gap}$.
\end{lstlisting}

\textsc{magma} code $2$:
\begin{lstlisting}[
  mathescape,
  columns=fullflexible,
  basicstyle=\ttfamily,
]
A6:=AlternatingGroup(6);
NN:=SmallGroups(720);
for i in [1..#NN] do
N:=SmallGroup(720,i);
  if not IsSolvable(N) then // for N to contain a copy of A6 necessarily N is insolvable
  NorSub:=NormalSubgroups(N);
  AA:=[A:A in NorSub|IsIsomorphic(A`subgroup,A6)];
    if not IsEmpty(AA) then
    i;
    end if;
  end if;
end for;
\end{lstlisting}

\textsc{magma} code $2$ output:
\begin{lstlisting}[
  mathescape,
  columns=fullflexible,
  basicstyle=\ttfamily,
]
763 // SmallGroup(720,763) = S6
764 // SmallGroup(720,764) = PGL(2,9)
765 // SmallGroup(720,765) = M10
766 // SmallGroup(720,766) = A6 x C2
// The identifications of SmallGroup(24,i) here may be checked using
the StructureDescription() or IdGroup() command in $\textsc{gap}$.
\end{lstlisting}
\end{landscape}

\end{document}